\documentclass{AIMS}
\usepackage{mathpazo}
                         \usepackage{amsthm}

\usepackage[disable,colorinlistoftodos,bordercolor=orange,backgroundcolor=orange!20,linecolor=orange,textsize=scriptsize]{todonotes}

\usepackage{graphics} %
\usepackage{amsmath} %
\usepackage{amssymb}  %
\usepackage{mathtools}
\usepackage{csquotes}
\renewcommand{\geq}{\geqslant}
\renewcommand{\leq}{\leqslant}

\usepackage{pgfplots, pgfplotstable}
\usepgfplotslibrary{statistics}

\usepackage{hyperref}
\usepackage{cleveref}

\crefname{table}{Table}{Tables}
\Crefname{table}{Table}{Tables}

\newif\iftrop
\troptrue

\newif\ifintro
\introtrue

\newtheorem{theorem}{Theorem}
\newtheorem{proposition}[theorem]{Proposition}
\newtheorem{corollary}[theorem]{Corollary}

\theoremstyle{remark}
\newtheorem{remark}[theorem]{Remark}
\theoremstyle{definition}
\newtheorem{definition}[theorem]{Definition}

\newcommand{\A}{\mathcal{A}}

\newcommand{\F}{\mathcal{F}}

\newcommand{\C}{\mathcal{C}}

\newcommand{\p}{[ p ]}

\newcommand{\R}{\mathbb{R}}

\DeclareMathOperator{\interior}{int}

\newcommand{\hT}{\widehat{T}}

\newcommand{\rnp}{\R_{+}^n}

\newcommand{\norm}[1]{\|#1\|}
\newcommand{\scal}[2]{\langle {#1} , {#2} \rangle}

\title[A convergent hierarchy to compute the joint spectral radius]{
A convergent hierarchy of non-linear eigenproblems to compute the joint spectral radius of nonnegative matrices}

\author{St\'ephane Gaubert and Nikolas Stott}
\thanks{INRIA and CMAP, {\'E}cole Polytechnique, CNRS, 91128 Palaiseau Cedex, France. \\  e-mail: {\tt\footnotesize stephane.gaubert@inria.fr}, {\tt\footnotesize nikolas.stott@polytechnique.edu}}
\thanks{St\'ephane Gaubert and Nikolas Stott were partially supported by the ANR projects MALTHY (ANR-13-INSE-0003) and DEMOCRITE (ANR-13-SECU-0007), by the PGMO program of EDF and Fondation Math\'ematique Jacques Hadamard, and by
the ``Investissement d'avenir'', r{\'e}f{\'e}rence ANR-11-LABX-0056-LMH.
S. Gaubert also gratefully acknowledges the support of the Mittag-Leffler institute. 
}%



\begin{document}

\maketitle
\thispagestyle{empty}

\begin{quotation}
To Fr\'ed\'eric Bonnans, at the occasion of his 60th birthday.

\end{quotation}
\begin{abstract}
We show that the joint spectral radius of a finite collection of
nonnegative matrices can be bounded by the eigenvalue of a
non-linear operator. This eigenvalue coincides with the
ergodic constant of a risk-sensitive control problem, or of an entropy
game, in which the state space consists of all switching
sequences of a given length. We show that, by increasing this length,
we arrive at a convergent approximation scheme to compute the joint spectral radius.
 The complexity of this method is exponential in the length of the
   switching sequences, but it is quite insensitive to the size
of the matrices, allowing us to solve very large scale instances
(several matrices in dimensions of order 1000 within a minute).
 An idea of this method is to replace a hierarchy of
 optimization problems,
introduced by Ahmadi, Jungers, Parrilo and Roozbehani,
by a hierarchy of nonlinear eigenproblems. To solve
the latter eigenproblems, we introduce
a projective version of Krasnoselskii-Mann iteration.
This method is of independent interest as it applies
more generally to the nonlinear eigenproblem for a monotone positively
homogeneous map. Here, this method allows for scalability by 
avoiding the recourse to
 linear or semidefinite programming
 techniques.
\end{abstract}

\section{Introduction}
\subsection{Motivation}
A fundamental issue, in optimal control, is to develop 
efficient numerical schemes that provide globally optimal
solutions. Dynamic programming
does provide a guaranteed global optimum but it 
is subject to the well known curse of dimensionality.
Indeed, the main numerical methods,
including monotone finite difference
or semi-Lagrangean schemes~\cite{crandall-lions,capuzzodolcetta,falcone-ferretti,carlini-falcone-ferretti},
and the anti-diffusive schemes~\cite{zidani-bokanowski},
are grid-based.
It follows that the time needed to obtain an approximate solution with a given accuracy is exponential in the dimension of the state space.

Recently, some innovative methods have been introduced in optimal control,
which somehow attenuate the curse of dimensionality, for structured
classes of problems. 

McEneaney considered in~\cite{mceneaney07}
hybrid optimal control problems in which a discrete
control allows one to switch between different linear quadratic
models. The max-plus type method that he introduced
approximates the value function by a supremum
of quadratic forms. Its complexity, which is
exponential in some parameters, has the remarkable
feature of being polynomial in the dimension~\cite{mccomplex,qusico}.
To produce approximations of the value function as concise as possible,
the method makes an intensive use of semidefinite programming~\cite{qucdc}.

A different problem consists in 
computing the joint spectral radius of a finite set of matrices~\cite{Jun09}.
This boils down to computing an ergodic value function,
known as the Barabanov norm. Specific numerical
methods have been developed, which approximate the Barabanov ball by a polytope~\cite{Guglielmi2014}, or are of semi-Lagrangean type~\cite{Kozyakin10}.
Ahmadi, Jungers, Parrilo and Roozbehani~\cite{pathcomplete} developed
a new method, based on a path complete automaton.
It approximates the Barabanov norm by a supremum of quadratic norms. 
Whereas the worst case complexity
estimates in~\cite{pathcomplete} are still subject to a curse of dimensionality,  in practice, the efficiency of the method is determined
by the complexity of the optimal switching law rather than 
by the dimension itself.
This allows one to solve instances of dimension inaccessible by a grid-based method.

In the max-plus method of McEneaney, and in the method of Ahmadi et al.,
solving large scale semidefinite programs appears to be the bottleneck,
limiting the applicability range.

In our recent work~\cite{cdc17,stottphd}, we introduced a new method to approximate
the joint spectral radius. We replaced the solution of large scale SDP problems
by the solution of eigenproblems involving non-linear operators,
the ``tropical Kraus maps''. The latter are
the analogues of completely positive maps,
or of ``quantum channels'' acting
on the space of positive semidefinite matrices, the operation
of addition being now replaced by a multivalued supremum operation in
the L\"owner order. To solve these eigenproblems, we used
iterative power type schemes, allowing us to deal with large scale instances (the algorithm of~\cite{cdc17,stottphd} could handle
several matrices of order 500 in a few minutes). 
The convergence of these iterative schemes, however, is only
guaranteed so far under restrictive assumptions, since the
``tropical Kraus maps'' are typically nonmonotone and expansive
in the natural metrics.

\subsection{Contribution}
In this paper, we develop a non-linear fixed point approach
to approximate the joint spectral radius in the special
case of {\em nonnegative matrices}.
We exploit a result of Guglielmi and Protasov~\cite{gugl_mono}, showing that for nonnegative matrices, it suffices to look for a {\em monotone norm}. 
We show here that such a monotone norm can be approximated by a finite supremum of linear forms, which are found as the solution of a non-linear eigenproblem.
This is in contrast
to earlier polyhedral approximation schemes,
relying for instance on linear programming. 

More precisely, we 
introduce a hierarchy of linear eigenproblems, parametrized
by a certain ``depth'', inspired by~\cite{pathcomplete,cdc17}, and we show that,
as the depth tends to infinity, the non-linear eigenvalue does converge
to the joint spectral radius. We remark that the initial (``depth 0'')
bound in our hierarchy coincides with the bound
of the joint spectral radius introduced by Blondel and Nesterov~\cite{blondel}.

The non-linear operator arising in our construction actually belongs to a 
known class: it can be identified to the dynamic programming operator
of an ergodic risk sensitive control problems~\cite{anantharam},
or of a (one player) ``entropy game''~\cite{asarin,akian_et_al:LIPIcs:2017:7026}. This operator enjoys remarkable properties, like log-convexity,
monotonicity, nonexpansiveness with respect to Thompson's part
metric or Hilbert's projective metric. As a result, 
computing the non-linear eigenvalue is a tractable problem.
It is shown to be polynomial time in~\cite{akian_et_al:LIPIcs:2017:7026}.
Moreover, large scale instances can be solved by power-type schemes. In particular, we introduce a projective version of the Krasnoselskii-Mann iteration.
We present this numerical scheme in a more general setting, for a monotone
positively homogeneous maps on the standard orthant. The convergence
of this scheme is obtained as a corollary of the convergence
of the original scheme. This projective scheme for nonlinear eigenproblems
may be of wider interest and applicability:
it has universal convergence properties and explicit bounds independent of the dimension, unlike the nonlinear power algorithm
considered classically, see e.g.~\cite{nussbaum88,agn12,FGH09}. 
It also has a geometric convergence property,
under less restrictive assumptions.

We report numerical results on large scale instances, up to dimension 5000,
obtained by an OCaml implementation of the present algorithm.

A comparison with our companion work~\cite{cdc17} 
may help to appreciate the present approach:
it appears to be a ``dequantization''
of the non-linear fixed point approach of~\cite{cdc17}. By ``dequantization'',
we mean that we use here operators acting on the standard orthant, in contrast,
the operator in~\cite{cdc17} acts on the cone of positive
semidefinite matrices. Whereas the approach of~\cite{cdc17} is
more general, leading to a convergent approximation
scheme for any family of matrices, the present algorithm
only applies to families of {\em nonnegative matrices}. However,
it is experimentally 
faster, and it has stronger theoretical convergence guarantees. This
suggests that the joint spectral radius problem is easier
for nonnegative matrices.

\subsection{Organization of the paper}
In \Cref{sec-jsp}, we recall some basic results on Barabanov norms of nonnegative matrices. In~\Cref{sec-hierarchy}, we introduce the family of non-linear eigenproblems to approximate the joint spectral radius. We show that these eigenproblems are solvable, under an appropriate irreducibility condition. In \Cref{sec-convergence},  we show that the non-linear eigenvalues in this hierarchy do converge
to the joint spectral radius. The projective Krasnoselskii-Mann
iterative scheme is analysed in~\Cref{sec-solving}. 
Benchmarks are presented
in~\Cref{sec-benchmarks}.

\section{The joint spectral radius of nonnegative matrices}
\label{sec-jsp}

The joint spectral radius $\rho(\A)$ of a finite collection of $n \times n$ real matrices $\A = \{ A_1, \dots, A_p\}$ is defined by
\begin{align*}
\rho(\A) \coloneqq \lim_{k \to \infty} \max_{1 \leq i_1, \dots, i_k \leq p} \norm{ A_{i_1} \cdots A_{i_k} }^{1/k} \,.
\end{align*}

When the set of matrices $\A$ is irreducible (meaning that there is no nontrivial subspace of $\R^n$ that is left invariant by all matrices), 
a fundamental result by Barabanov~\cite{barab} shows that there is a norm $\nu$ on $\R^n$ such that
\begin{align}
\label{eq:invariantnorm}
\max_{1 \leq i \leq p} \nu(A_i x) = \lambda \nu(x) \;,\, \forall x \in \R^n \,,
\end{align}
for some positive real number $\lambda$. The scalar $\lambda$ is unique and coincides with the joint spectral radius $\rho(\A)$.

A norm that satisfies~\Cref{eq:invariantnorm} is called an \emph{invariant norm}. A norm that only satisfies the inequality 
\begin{align}
\label{e-lambdextrema}
\max_{1 \leq i \leq p} \nu(A_i x) \leq \lambda \nu(x)
\end{align}
for all vectors $x \in \R^n$ is called a $\lambda$-\emph{extremal norm}. In that case, it is readily seen that $\lambda \geq \rho(\A)$, so that $\lambda$-extremal norms provide safe upper bounds of the joint spectral radius.

We now assume that the matrices in $\A$ are nonnegative, i.e.~their entries take nonnegative values. It is then readily seen that all matrices in $\A$ leave the (closed) cone of nonnegative vectors invariant. The latter cone, denoted by $\rnp$,  induces an ordering on $\R^n$: we have $x \leq y$ if and only if $y - x$ is nonnegative. 
We note that a vector belongs to the interior of $\rnp$ if its entries are positive.
Recall that the cone $\rnp$ is self-dual, so that $x \leq y$ if and only if $\scal{u}{y-x} \geq 0$ for all $u \in \rnp$. 
This cone also induces a lattice structure on $\R^n$, meaning that the supremum of two vectors $x,y$ always exists and is given coordinate-wise by $\big[ \sup(x,y) \big]_i = \sup(x_i , y_i)$.
A norm defined on $\R^n$ is called {\em monotone} if $0\leq x \leq y$ 
implies $\nu(x) \leq \nu(y)$.
The (extreme) faces of $\rnp$ are the sets $\{ x \in \rnp \colon x_i = 0 \;\text{if}\; i\notin I \}$ for $I \subseteq \{ 1, \dots,n\}$. The cases $I = \emptyset$ (corresponding to $F = \{0\}$) and $I = \{ 1, \dots,n\}$ (giving $F = \rnp$) 
yield the {\em trivial} faces.
When the matrices in $\A$ are nonnegative, the irreducibility assumption  on $\A$ can be weakened to \emph{positive-irreducibility}, meaning that there is no non-trivial face of the cone of nonnegative vectors that is left invariant by all matrices in $\A$.
A theorem by Guglielmi and Protasov~\cite{gugl_mono} shows that, in this setting, the norm in~\Cref{eq:invariantnorm} can be chosen to be monotone.
\begin{theorem}[Corollary 1 in~\cite{gugl_mono}]
\label{thm:exis_mono}
A positively-irreducible family of nonnegative matrices has a monotone invariant norm.
\end{theorem}

We shall say that a map $\nu$ from $\rnp$ to $\R$ is a {\em monotone hemi-norm}
if it is convex and positively homogeneous of degree $1$, 
if $0\leq x\leq y$ implies $\nu(x)\leq \nu(y)$, and if $\nu(x)=0$ with $x\geq 0$
implies $x=0$. The term hemi-norm is borrowed to~\cite{GV10}, functions
of this kind are also known as weak Minkowski norms in metric
geometry~\cite{papadopoulos}

Note that a monotone hemi-norm $\nu$ defined on $\rnp$ 
extends to a monotone norm on $\R^n$:
\begin{align}
\label{eq:mon_norm_ext}
\widehat{\nu} (x) \coloneqq \inf\{ \nu(y) \vee \nu(z) \colon x = y - z\,,\text{with}\, y, z \geq 0 \} \enspace .
\end{align}
We shall say that $\nu$ is a {\em monotone $\lambda$-extremal hemi-norm} on $\rnp$
if
\begin{align*}
\max_{1 \leq i \leq p} \nu(A_i x) \leq \lambda \nu(x) \,,\;\forall x \in \rnp \,.
\end{align*}
This implies that the associated monotone norm $\widehat{\nu}$ is a $\lambda$-extremal norm. In this way, it suffices to study monotone $\lambda$-extremal hemi-norms defined on $\rnp$.

\section{A hierarchy of non-linear eigenproblems}
\label{sec-hierarchy}

\subsection{Definition of the operators}

In the sequel, we consider a finite set of $n\times n$
nonnegative matrices $\A = \{ A_1, \dots, A_p \}$. We denote by $\p = \{ 1, \dots, p \}$.

The operator considered at the $0$-level of the hierarchy is given by
\begin{align*}
T^0(x) \coloneqq \sup_{1 \leq a \leq p} A_a^Tx \,.
\end{align*}
Higher levels of the hierarchy are built by introducing a memory process that keeps track of the past matrix products, up to a given depth. More precisely, given an integer $d$, the operator considered in the $d$-level 
is a self-map of the product cone $\prod_{s \in \p^d} \rnp$.
It maps the vector $x = (x_s)_{s \in \p^d}$, where each $x_s\in \rnp$,
to the vector $T^d(x)$, whose $s$-component is the vector of $\rnp$ given by
\begin{align}
T^d_s(x) \coloneqq \sup_{r,a\colon \tau^d(r,a) = s} A_a^Tx_r \,.
\label{e-def-Td}
\end{align}
Here, the map $\tau^d\colon \p^d \times \p \to \p^d$ is the transition map of  the \emph{De Bruijn automaton} of length $d$ on $p$ symbols: given a word $i_1\cdots i_d\in \p^d$, we have
\begin{align*}
\tau^d( i_1\cdots i_d, a) = i_2\cdots i_d a \,.
\end{align*}
In other words, the transition forgets the initial symbol of a sequence,
and concatenates the letter $a$ representing the most recent switch, to this
sequence.

The map $T^d$ is monotone with respect to the cone $\prod_{s \in \p^d} \rnp$, 
i.e., 
\begin{align*}x\leq y\implies T^d(x)\leq T^d(y)\enspace, 
\end{align*}
for all $x,y\in \prod_{s \in \p^d} \rnp$, 
and it is (positively) homogeneous (of degree one), meaning that
\begin{align*}
T^d(\lambda x) = \lambda T^d(x)
\end{align*}
holds for all positive $\lambda$.

\subsection{Some results of non-linear Perron-Frobenius theory}
Monotone and homogeneous maps are studied
in non-linear Perron-Frobenius theory. We recall some basic
results here, referring the reader to~\cite{nussbaum88,nussbaumlemmens} for background.

The {\em spectral radius} of a monotone and homogeneous map $f$ defined on a cone $\C$, denoted by $r(f)$ is defined by:
\begin{align*}
r(f) \coloneqq \lim_{k \rightarrow +\infty} \norm{f^k(x)}^{1/k}
\end{align*}
for $x \in \interior \C$. This value is independent of the choice of $x$ and the norm $\norm{ \cdot}$ if the cone $\C$ is included in a finite
dimensional space, see~\cite{Nuss-Mallet,agn,nussbaumlemmens}.

We say that a monotone and homogeneous map $f\colon \rnp \to \rnp$ is
{\em positively irreducible} if it does not leave invariant a non-trivial face of $\rnp$. A basic result of non-linear Perron-Frobenius
theory, which follows as a consequence of Brouwer theorem, shows that
a positively irreducible map has an eigenvector in the interior
of the cone. Then, the associated eigenvalue $\lambda$ coincides with the spectral radius $r(f)$.
The same conclusion holds, in fact, under less demanding assumptions~\cite{gg04}, however, for the present class of operators, positive irreducibility will suffice.

Nussbaum proved in~\cite{nussbaum86} that the classical variational characterization of the Perron root of a nonnegative matrix carries over to the non-linear setting. The next theorem follows by combining results of \cite{nussbaum86}, ~\cite{gg04}
and~\cite{AGGut10}.

\begin{theorem}[Non-linear Collatz-Wielandt formul\ae~{\cite[Theorem~3.1]{nussbaum86}, \cite[Prop.~1]{gg04}, \cite[Lemma~2.8]{AGGut10}}]
\label{thm:cw}
Given a continuous, monotone and homogeneous map $f$ on the cone $\rnp$, we have
\begin{align*}
r(f)
&= \inf \big\{ \rho > 0 \colon \exists u \in \interior \rnp, \qquad f(u) \leq \rho u \big\} \,\\
&= \max \big\{\rho \geq 0 \colon \exists v \in \rnp\setminus\{0\}, \qquad f(v) \geq \rho v\} \enspace .
\end{align*}
\end{theorem}
We write ``inf'', as the infimum is not attained in general, whereas the maximum is always attained. 

In particular, if the map $f$ is not positively-irreducible, 
it may be the case that $f(u) = \lambda u$ holds for some nonzero
vector $u$ in the boundary of the cone $\rnp$. Then, we can only
conclude from the second of the Collatz-Wielandt formul\ae\ 
that $\lambda \leq r(f)$. However, by the first formula, we do have
$r(f)=\lambda$ if $u$ belongs to the interior of $\rnp$.

\subsection{Construction of the hierarchy}

For every integer $d \geq 0$, the $d$-level of the hierarchy consists in solving the non-linear eigenproblem:
\begin{align}
\label{hier}
\tag{$E_d$}
\begin{cases}
T^d(u) = \lambda_d u \\
u \in \prod_{s\in \p^d} \rnp\,, u\neq 0
\end{cases}
\end{align}

The first main result shows that every problem~\eqref{hier} has a solution, and that a solution provides an upper bound on the joint spectral radius $\rho(\A)$ and a corresponding monotone $\lambda_d$-extremal hemi-norm.

\begin{theorem}
\label{thm:over}
Suppose that the set of nonnegative matrices $\A$ is positively-irreducible.
Then Problem~\eqref{hier} has a solution. Any such solution $(\lambda_d, u)$ satisfies
\begin{align*}
 \rho(\A) \leq \lambda_d \leq r(T^d) \,.
\end{align*}
Moreover, the map $\norm{ x }_u \coloneqq \max_s \scal{u_s}{x} $ is a monotone $\lambda_d$-extremal hemi-norm:
\begin{align*}
\max_a \norm{A_a x}_u \leq \lambda_d \norm{x}_u \,,\forall x\in \rnp \,.
\end{align*}
\end{theorem}

\begin{proof}
First, note that the map $T^d$, which is continuous and positively
homogeneous on the cone $\prod_{s \in \p^d} \rnp$, has an eigenvector $u$,
i.e., $T^d(u) = \lambda_d u$ for some $\lambda \geq 0$. This is indeed
 a standard result, which follows by applying Brouwer fixed point theorem
to the map $x\mapsto T^d(x)/\|T^d(x)\|_{1}$, where $\|\cdot\|$ denotes
the $\ell_1$ norm. This map sends continuously
the simplex $\Delta:= \{x\in \prod_{s \in \p^d} \rnp\mid \|x\|_1=1\}$ to itself.

Let us write $u= (u_s)_s$ this eigenvector, and for each $s$, let $F_s$ denote the minimal
face of $\R_+^n$ containing the vector $u_s$.
We introduce the set $\F \coloneqq \sum_{s \in \p^d} F_s$. The latter set is a face of $\rnp$ and satisfies $A_a \cdot \F \subseteq \F$ for all $a \in \p$ by definition of the map $T^d$, hence $\F = \rnp$ since the set of matrices $\A$ is irreducible. It follows that the vector $\sum_{s\in \p^d} u_s$ is positive and that the map $x \mapsto \max_s \scal{u_s}{x}$ is a monotone hemi-norm on $\rnp$.

We have $\scal{A_a^Tx}{u_r} \leq \scal{x}{T^d_{\tau(r,a)}(u)} \leq \lambda_d \scal{x}{u_{\tau(r,a)}}$ (we write $\tau$ instead of $\tau^d$).
 Taking the supremum over $r$ and $a$, we arrive at $\max_{a} \norm{A_ax}_u \leq \lambda_d \norm{x}_u$, hence $\rho(\A) \leq \lambda_d$.
We deduce from the second of the Collatz-Wielandt formul\ae\ in~\Cref{thm:cw},  that $\lambda_d \leq r(T^d)$.
\end{proof}

The positive-irreducibility of $T^d$ can be decided by checking whether a lifted version of the set of matrices $\A$ is positively irreducible. In the following, the set $\{ e_r \colon r \in \p^d\}$ denotes  the canonical basis of the space $\R^{p^d}$ and $\otimes$ is the Kronecker product.

\begin{proposition}
\label{prop:pos_irr}
The map $T^d$ is positively-irreducible if and only if the set of matrices $\{ (e_re_s^T) \otimes A_a \colon \tau^d(r,a) = s, \;r,s\in \p^d\,, a \in \p \}$ is positively irreducible.
\end{proposition}
\begin{proof}
First, we consider the case $d = 0$. Suppose the map $T^0$ is not positively irreducible, that is, there is a non-trivial face $F$ of the cone $\rnp$ that is invariant by $T^0$. 
Given any $x \in F$, we have $A_a^T x \leq T^0(x) \in F$, thus $F$ is also invariant by all matrices $A_a^T$, which is equivalent to saying that the set $\A$ is not positively irreducible.

In the general case, we can rewrite the map $T^d$, originally defined on the space $\prod_{s \in \p^d} \R^n$, on the space $\R^{p^d} \otimes \R^n$. To this end, we "stack" the (vector) components of the vector $(x_r)_{r \in \p^d}$ as one vector $l[x] \coloneqq \sum_{r} e_r \otimes x_r$.

Moreover, we obtain by "stacking" the components of $T^d(x)$ the vector
\begin{multline*}
l\Big[T^d(x)\Big]  = 
\sum_s  \; e_s  \otimes \big[T^d(x)\big]_s
 = \sup_s \;  e_s \otimes \Bigg[\sup_{r,a\colon \tau^d(r,a) = s}  \Big[ A_a^T x_r \Big] \Bigg] \\
= \sup_{r,s,a\colon \tau^d(r,a) = s} \Big[ e_s \Big] \otimes \Big[ A_a^T x_r \Big] 
= \sup_{r,s,a\colon \tau^d(r,a) = s} \sum_t \Big[ e_se_r^Te_t \Big] \otimes \Big[ A_a^T x_t \Big]  \\
 = \sup_{r,s,a \colon \tau^d(r,a) = s} \Big[ (e_re_s^T)^T \otimes A_a^T  \Big] \Big[ \sum_t e_t \otimes x_t \Big] 
 = \sup_{r,s,a} A_{r,s,a}^T \; l[x]\,,
\end{multline*}
where we have used the fact that the coefficients of $e_s$ besides position $s$ are zero,   that $e_r^Te_t = 0$ if $r \neq t$ and we have denoted $A_{r,s,a} = (e_re_s^T) \otimes A_a$ when $\tau^d(r,a) = s$, and $A_{r,s,a} = 0_{n,n}$ otherwise.

Hence, the map $T^d$ is positively irreducible if and only if the map
\begin{align*}
y \mapsto \sup_{r,s,a} A_{r,s,a}^T y
\end{align*}
 is positively irreducible. This reduces to the case $d = 0$ where the set of matrices $\A$ is replaced by the set $\{ A_{r,s,a} \}$.
\end{proof}

The term ``hierarchy'' for the sequence of problems~\eqref{hier} is 
justified by the following proposition.

\begin{proposition}
Suppose that the set of nonnegative matrices $\A$ is positively-irreducible.
Then $r(T^{d+1}) \leq r(T^d)$ for all $d$.
\end{proposition}
\begin{proof}
Let $u \in \interior \prod_{s \in \p^d} \rnp$ and $\lambda$ a positive real number such that $T^d(u) \leq \lambda u$ for some $d$.
Let $v$ denote the vector in $\prod_{s \in \p^{d+1}} \rnp$ defined for all $r\in \p^d$ and $a\in \p$ by
$
v_{ar} \coloneqq u_r 
$.
We have
\begin{align*}
A_b^T v_{ar} =  A_b^T  u_r 
\leq \lambda u_{\tau^d(r,b)}
 = \lambda v_{\tau^{d+1}(ar,b)} \,.
\end{align*}
 Taking the supremum over $r$ and $a$, we obtain $T^{d+1}(v) \leq \lambda v$.
Each vector $v_{ar}$ is positive, so taking the infimum over $\lambda$, 
by the first of the Collatz-Wielandt formul\ae\ in~\Cref{thm:cw}, we arrive at $r(T^{d+1})\leq r(T^d)$.
\end{proof}

\section{Convergence of the hierarchy of nonlinear eigenproblems}
\label{sec-convergence}

The next theorem shows that the spectral radius of the map $T^d$ approximates the joint spectral radius $\rho(\A)$ up to a factor $n^{1/(d+1)}$.
The proof of this result is inspired by the ones found in~\cite{pathcomplete,matthew} in the case of piecewise quadratic approximations of norms. The latter proofs rely on the approximation of a symmetric convex body by the \emph{L\"owner-John ellipsoid}. Here, we use the fact that a monotone hemi-norm $\nu$ can be approximated by a monotone linear map, up to a factor $n$, as shown
by the following observation. 

\begin{proposition}
\label{prop:c_mono}
Given a monotone hemi-norm $\nu$, there is a vector $c$ with positive entries such that
\begin{align*}
\nu(x) \leq \scal{c}{x} \leq n \nu(x) \,,
\end{align*}
for all nonnegative vectors $x$.
\end{proposition}
\begin{proof}
Let $c$ denote the vector defined by $c_i \coloneqq \nu(e_i)$, with $(e_i)_{1\leq i \leq n}$ the canonical basis of $\R^n$, and observe that $c_i$ is positive
since $\nu$ is a hemi-norm. Let $x$ denote a nonnegative vector.
By convexity and homogeneity of the map $\nu$, we have $\nu(x) \leq \sum_i x_i \nu(e_i) = \scal{c}{x}$.
By monotonicity of $\nu$, we have $x_i c_i = \nu(x_i e_i) \leq \nu(x)$, hence $\scal{c}{x} \leq n \nu(x)$.
\end{proof}

We now show that the sequence of approximations provided by this hierarchy do converge when $d$ tends to infinity.

\begin{theorem}
\label{thm:upper}
Suppose that the set of nonnegative matrices $\A$ is positively-irreducible. Then
\begin{align*}
r(T^d) \leq n^{1/(d+1)}\rho(\A)\,.
\end{align*}
\end{theorem}

\begin{proof}
We first prove the case $d=0$, in which case the map $T^0$ is positively-irreducible. By~\Cref{thm:exis_mono}, the set $\A$ admits a monotone invariant hemi-norm denoted by $\nu$.
 As noted earlier,
By~\Cref{prop:c_mono} 
  there is a positive vector $c$ such that $\scal{c}{x} \leq \nu(x) \leq n\scal{c}{x}$ holds for all nonnegative vectors $x$. Since $\nu$ is a monotone invariant hemi-norm, we have $\scal{A_a^Tc}{x}=\scal{c}{A_ax} \leq \nu(A_ax) \leq \rho(\A) \nu(x) \leq n\rho(\A) \scal{c}{x}$ for all nonnegative vectors $x$, hence 
$\scal{n\rho(\A) c-A_a^T c}{x}\geq 0$ for all such vectors $x$, which implies
that $n\rho(\A) c-A_a^T c\geq 0$, i.e.,  $A_a^T c\leq n\rho(\A)$.
Taking the supremum over $a$, we get $T^0(c) \leq n\rho(\A) c$. Thus $r(T^0) \leq n\rho(\A)$ by~\Cref{thm:cw}.

We now prove the general case. It will be convenient to consider the variant of the map $T^d$ obtained by replacing the set of matrices $\A$ by the set $\A^{d+1}$ of products in $\A$ of length $d+1$, yielding the map on $\R^n$ given by:
\begin{align*}
\widehat{T}^d(x)  = \sup_{ M \in \A^{d+1}} M^Tx
\end{align*}
Now, let $v$ denote a positive vector in $\rnp$ and $\mu$ a positive real number such that $\widehat{T}^d(v) \leq \mu v$. By~\Cref{thm:cw}, the infimum of such real numbers $\mu$ is equal to the spectral radius $r(\widehat{T}^d)$. We introduce the collection of vectors $u = (u_s)_{s \in \p^d}$ defined by
\begin{align}
u_s = \sum_{0 \leq k \leq d} \mu^{-k/(d+1)} 
A^T_{s(1)} \cdots A^T_{s(k)} 
 v \,,
\label{e-note}
\end{align}
where $s(i)$ denotes the $i$-th letter of the word $s \in \p^d$,
with the convention that the product in the summation~\eqref{e-note}
is equal to $v$ when $k=0$.
It is readily seen that this collection satisfies the set of inequalities $A^T_a u_r \leq \mu^{1/(d+1)} u_{\tau(r,a)}$.
Moreover, by definition of $u_s$, $u_s\geq v$, and so $u_s$ is positive.
We deduce that $T^d(u) \leq \mu^{1/(d+1)} u$, hence $r(T^d) \leq r(\widehat{T}^d)^{1/(d+1)}$ by~\Cref{thm:cw}.

It remains to be shown that $r(\hT^d) \leq n \rho(\A)^{d+1}$. Now, we consider the perturbed map $\hT^d_\varepsilon$ defined for $\varepsilon > 0$ by
\begin{align*}
\hT^d_\varepsilon(x) \coloneqq \sup \big\{ M^Tx \colon M \in \A_\varepsilon^{d+1} \big\} \,,
\end{align*}
originating from the family of perturbed matrices $\A_\varepsilon \coloneqq \{ A_a + \varepsilon J \}_{a \in \p}$, 
where $J$ is the square matrix with all entries equal to $1$.
 The matrices in $\A_\varepsilon$ are positive hence the latter set is positively-irreducible. Thus we fall in the case $d=0$ and we obtain 
 \begin{align*}
r(\hT^d_\varepsilon) \leq \rho(\A_\varepsilon^{d+1}) = \rho(\A_\varepsilon)^{d+1} \,.
 \end{align*}
 Moreover, the inequality $\hT^d(x) \leq \hT^d_\varepsilon(x)$ holds for all nonnegative vectors $x$, hence $r(\hT^d) \leq r(\hT^d_\varepsilon)$ by~\cite{nussbaumlemmens}. Combined with the fact that $\lim_{\varepsilon \to 0} \rho(\A_\varepsilon) = \rho(\A)$ as proved in~\cite{jungers}, we obtain the desired inequality.
\end{proof}

We obtain as an immediate corollary of~\Cref{thm:over,thm:upper} that the hierarchy is convergent, in the sense that any sequence of eigenvalues of the map $T^d$ converges towards the joint spectral radius.

\begin{corollary}
Suppose that the set of nonnegative matrices $\A$ is positively-irreducible. If $\lambda_d$ denotes an eigenvalue of the map $T^d$ for all $d$, then
\begin{align*}
\lim_{d \to \infty} \lambda_d = \rho(\A) \,.
\end{align*}
In particular, the sequence of spectral radii $r(T^d)$ is non-increasing and its limit is equal to $\rho(\A)$.
\end{corollary}

\section{Solving the non-linear eigenproblem}
\label{sec-solving}

Several numerical methods allow one to solve the nonlinear eigenproblem~\eqref{hier}. First, the log-convexity property of $T^d$ allows a reduction to convex programming, which entails a polynomial time bound (see for instance the part of~\cite{akian_et_al:LIPIcs:2017:7026} concerning ``Despot free'' entropy games).  
There are also algorithms, more efficient in practice, that do not have
polynomial time bounds. 
Protasov proposed a ``spectral simplex'' algorithm~\cite{protasov}.
A policy iteration scheme was proposed in~\cite{akian_et_al:LIPIcs:2017:7026}.
The spectral simplex, like policy iteration, involves at each step
the computation of the spectral radius of a nonnegative matrix, which
is generally the bottleneck.

For the huge scale instances which are of interest here, it is more convenient to employ a simpler iterative scheme. We propose to use 
a projective version of the Krasnoselskii-Mann iteration~\cite{mann,krasno}.
The Krasnoselskii-Mann iteration can be written as $x_{k+1}=2^{-1}(x^k +F(x^k))$,
it was originally considered when $F$ is a nonexpansive mapping $F$ 
acting on a uniformly convex Banach space~\cite{krasno}. 
The uniform convexity assumption was relaxed by Ishikawa:
\begin{theorem}[{\cite[Theorem 1]{ishikawa}}]\label{ishikawa}
Let $D$ be a closed convex subset of a Banach space $X$, let $F$
be a nonexpansive mapping sending $D$ to a compact subset of $D$.
Then, for any initial point $x^0\in D$, the sequence defined by 
$x_{k+1}=2^{-1}(x^k +F(x^k))$ 
converges to a fixed point of $F$. 
\end{theorem}
The analysis of this iteration, by Edelstein and O'Brien~\cite{obrien},
involves the notion of {\em asymptotic regularity}. The latter property
means that $\|F(x^k)-x^{k}\|$ tends to $0$ as $k$ tends to
infinity. The estimate $\|F(x^k)-x^{k}\|$ provides a convenient
way to measure the convergence. 
Baillon and Bruck obtained in~\cite{baillonbruck} the following
quantitative asymptotic regularity estimate
\begin{align}
\|F(x^k)-x^{k}\|\leq \frac{2\,\operatorname{diam}(D)}{\sqrt{\pi k}} \enspace,
\label{baillonbruck}
\end{align}
see~\cite{cominetti} for more information.
Observe that the rate $1/\sqrt{k}$ is independent of the dimension.

Here, we adapt the idea of the Krasnoselskii-Mann iteration
to the eigenproblem, by considering it as a fixed
point problem in the projective space. 

It will be convenient to consider an arbitrary
monotone and positively
homogeneous (of degree one) $f: \R_+^N \to \R_+^N$, 
having a positive eigenvector
$u$. In the present application, we will
consider the special case $f:=T^d$, however, the scheme does
converge in a rather general setting, and it may have
other applications.
\begin{definition}[Projective Krasnoselskii-Mann iteration]
Starting from any positive vector $v^0\in \R_+^N$ such that $\prod_{i\in [N]}v^0_i=1$, compute the sequence defined by
\begin{align}
\label{eq:krasn}
v^{k+1} & = \Bigg[ \frac{f(v^k) }{G\big[f(v^k) \big]} \circ v^{k} \Bigg]^{1/2} \,,
\end{align}
where $\circ$ denotes the entrywise product of two vectors.
and $G(x) = (x_1 \cdots x_N)^{1/N}$ denotes the geometric mean of the components of the vector $x$.
\end{definition}
By comparison with the original Krasnoselskii-Mann iteration,
the arithmetic mean is replaced by the geometric
mean, and a normalization is introduced to deal with the projective
setting. 

To show that this iteration does converge, we need
to recall some metric properties of monotone positively homogeneous maps.
We shall use a seminorm called {\em Hopf's oscillation}~\cite{Bus73}
or {\em Hilbert's seminorm}~\cite{gg04}. The latter
is defined on $\R^N$ by
\[ \norm{x}_H = \inf\{\beta - \alpha \colon \alpha,\beta\geq 0, \; \; \alpha e \leq x \leq \beta e \}
\enspace ,\]
with $e = (1 \, \cdots \, 1)^T$.
This seminorm is invariant by addition with a constant ($\norm{x + \alpha e}_H = \norm{x}_H$). Observe that Hopf's
oscillation defines a norm on the vector space $X:=\{x\in \R^N\colon \sum_i x_i=0\}$.
The {\em Hilbert's projective metric}~\cite{nussbaum88}, defined on the interior of the cone $\R_+^N$, is given by
\[ d_H(x,y) = \|\log x -\log y\|_H \enspace ,
\]
where $\log$ is understood entrywise, meaning that 
$\log x:= (\log x_i)_{i\in [N]}$. 
Observe that
$d_H(\alpha x, \beta y)=$ $d_H(x,y)$ for all $\alpha,\beta>0$; so
$d_H$ defines a metric on the space of rays included in the interior
of the cone. We shall also use {\em Thompson's metric}, defined
on the interior of $\R_+^N$ by 
\[
d_T(x,y) = \|\log x -\log y\|_\infty \enspace ,
\]
where $\|\cdot\|_\infty$ is the sup-norm. It is known that if $f$ is monotone
and positively homogeneous, and if it preserves the interior
of $\R_+^N$, then it is nonexpansive both in Hilbert's
projective metric and in Thompson's metric, see e.g.~\cite[Lemma~4.1]{agn12}.

We next show that the scheme~\eqref{eq:krasn} does converge.
\begin{theorem}\label{th-converge}
Suppose that $f$ is a monotone positively homogeneous map $\R_+^N\to \R_+^N$
having a positive eigenvector. Then,
the iteration in~\eqref{eq:krasn} initialized at any positive vector
$v^0\in \R_+^N$ such that $\prod_{i\in [N]}v^0_i=1$, 
converges towards an eigenvector of $f$,
and $G(f(v^k))$ converges to $r(f)$. 
\end{theorem}
We next show that this reduces to the convergence of the original scheme,
after a suitable transformation.
\begin{proof}
Let $u$ denote a positive eigenvector of $f$,
so that
$f(u)= \lambda u$
for some $\lambda >0$.  
For all $x$ in the interior of $\R_+^N$, we can write $\alpha u\leq x$
for some $\alpha>0$, and since $f$ is order preserving and positively homogeneous, we deduce that $f(x)\geq \alpha f(u) =\alpha u$, so $f$ preserves
the interior of $\R_+^N$. 
We now define the self-map $S$ of $\R^N$ by
\begin{align}
S(y) = \log  f \big[\exp (y) \big] 
\enspace ,\label{e-def-S}
\end{align}
where, again, the notation $\log$ for a vector is understood entrywise,
and similarly for $\exp$.

The map $S$ is monotone and commutes with the addition of a constant. It follows that the map $S$ and the map 
\begin{align}
\widehat{S} \colon y \mapsto S(y) - N^{-1}\scal{e}{S(y)} e
\label{e-def-Shat}
\end{align}
are also non-expansive with respect to Hopf's oscillation, see e.g.~\cite[Lemma~4.1]{agn12}.

We also note that $S(\log u) = \log u + (\log \lambda) e$.

Given $r > 0$, we consider $B_r \coloneqq \{ x \in \R^N \colon \norm{x-\log u}_H \leq r \}$.
 The set $B_r$ is invariant by the map $S$, since, using the nonexpansiveness
of this map in Hopf's oscillation
\begin{align*}
\norm{S(x)- \log u}_H
& =\norm{S(x) - \log u - (\log(\lambda+1)) e }_H \\
& =\norm{S(x) - S( \log u) }_H \\
& \leq \norm{x- \log u}_H \,.
\end{align*} 
The same holds for the map $\widehat{S}$ since it only differs from $S$ by addition of a multiple of the vector $e$.
Moreover, the vector $\log v^0$ belongs to $B_r$ for $r$ large enough. We fix such an $r$ in the sequel.
Observe that $\widehat{S}$ is a nonexpansive self-map of the normed
space $X$ equipped with Hopf's oscillation, and that this map leaves
invariant the 
set $B_r \cap X$. The latter set is closed and bounded in the Euclidean metric, hence it is compact. It is also convex. By~\Cref{ishikawa}, the iterative process defined by
\begin{align}
\label{eq:iter_log}
y^{k+1} = \frac{1}{2} \widehat{S}(y^k) + \frac{1}{2} y^k \enspace,
\end{align}
initialized at any point $y^0\in B_r\cap X$,
converges towards some vector $y \in B_r\cap X$. This limit satisfies $\widehat{S}(y) = y$.

By writing $v = \exp(y)$ and $v^k = \exp(y^k)$, we rewrite~\Cref{eq:iter_log} to obtain the iteration in~\Cref{eq:krasn}, and observe
that the condition $\prod_{i\in [N]}v^0_i =1$ entails $y^0\in X$.
Hence, the sequence $v^k$ converges to $v$. Recall that $f$ is nonexpansive in Thompson's metric and that
the latter induces in the interior of the cone $\R_+^N$ the euclidean topology.
It follows that $f$ is continuous,
with respect to this topology, on the interior of the cone. Hence, passing
to the limit in~\eqref{eq:krasn}, we obtain 
\[
v= 
\Bigg[ \frac{f(v) }{G\big[f(v) \big]} \circ v \Bigg]^{1/2} \,,
\]
and so $f(v)=\mu v$ with $\mu:= G[f(v)]$.
\end{proof}

The following quantitative version of \Cref{th-converge} shows that $f(v^k)$ becomes approximately proportional
to $v^k$ as $k$ tends to infinity, i.e., $v^k$ is an ``approximate eigenvector''. 
\begin{corollary}
Suppose that $f$ is a monotone positively homogeneous map $\R_+^N\to \R_+^N$
having a positive eigenvector. Then,
the sequence
$v^0,v^1,\dots$ constructed by the projective Krasnoselskii-Mann iteration
satisfies
\begin{align}
d_H(f(v^{k}),v^k) \leq \frac{4}{\sqrt{\pi k}} d_H(v^0,u) \enspace .
\label{e-bound}
\end{align}
\end{corollary}
\begin{proof}
Since $d_H(\alpha x, \beta y)=d_H(x,y)$ for all $\alpha,\beta>0$, and since
$f$ is positively homogeneous, we may assume that $\sum_i v^0_i =1$, and
so $\log v^0\in X=\{x\in \R^N\colon \sum_i x_i=0\}$.
Let us now choose $D:= B_r\cap X$ with $r=d_H(v^0,u)$, so that $v^0\in D$. Then,
it follows from the final part of the proof of \Cref{th-converge} 
that $\hat{S}$ leaves invariant $D$. Then, the inequality~\eqref{e-bound} follows
from~\eqref{baillonbruck}. 
\end{proof}
We also deduce that the projective Krasnoselskii-Mann iteration
provides convergent lower and upper
approximations of the spectral radius
$r(T^d)$.
\begin{corollary}
\label{cor:lambda_control}
Suppose that $f$ is a monotone positively homogeneous map $\R_+^N\to \R_+^N$
having a positive eigenvector $u$. 
Let
$v^0,v^1,\dots$ be the sequence constructed by the projective Krasnoselskii-Mann iteration, and let
\[
\alpha_k := \min_{i\in [N]} \frac{[f(v^k)]_i}{v^k_i},
\qquad 
\beta_k := \max_{i\in [N]} \frac{[f(v^k)]_i}{v^k_i},
 \enspace .
\]
Then,
\begin{align}
\alpha_k \leq r(f) \leq \beta_k 
\label{e-sandwitch}
\end{align}
and
\begin{align}
\log \beta_k -\log \alpha_k 
\leq
\frac{4}{\sqrt{\pi k}} d_H(v^0,u) \enspace .
\label{e-bound2}
\end{align}
\end{corollary}
\begin{proof}
By definition of Hilbert's projective metric, we have
\[ \log(\beta_k/\alpha_k) = d_H(f(v^k),v^k)\enspace ,
\]
and 
\[\alpha_k v^k\leq f(v^k)\leq \beta_k v^k \enspace .
\]
Then, ~\eqref{e-sandwitch} follows from the Collatz-Wielandt formula (\Cref{thm:cw}), whereas~\eqref{e-bound2} follows from~\eqref{e-bound}. 
\end{proof}

\begin{corollary}
Let $f \coloneqq T^d$ and suppose that $T^d$ has a positive eigenvector $u$. Then the sequence $(\beta_k)_k$ defined in~\Cref{cor:lambda_control} satisfies
\begin{align}
\log \beta_k - \log \rho(\A) \leq \frac{4}{\sqrt{\pi k}} d_H(v^0, u) + \frac{\log n}{d+1} \,.\label{finalbound}
\end{align}
\end{corollary}
\begin{proof}
We combine the inequalities in~\Cref{thm:upper} and~\Cref{cor:lambda_control}.
\end{proof}

\begin{remark}
One can give an a priori bound on the vector $u$,
to get an explicit control of $d(v^0,u)$ in~\eqref{finalbound}. See
Lemma~16 of~\cite{akian_et_al:LIPIcs:2017:7026}. 
\end{remark}
\begin{remark}
By~\Cref{prop:pos_irr} and the Perron-Frobenius theorem, the map $T^d$ has an eigenvector with positive entries when the set of matrices given in~\Cref{prop:pos_irr} is positively-irreducible.
We point out that the same iteration also converges under the weaker assumption that $\A$ is positively-irreducible, but it must be initialized with a vector belonging to the interior of a non-trivial face of the cone $\prod_{s \in \p^d} \rnp$ that is invariant by $T^d$ and that has minimal dimension.
\end{remark}

We now show that the projective Krasnoselskii-Mann
iteration converges at a geometric rate under an additional assumption.
Let $u$ denote
a positive eigenvector of $f$, so that $f(u)=\lambda u$ with $\lambda>0$,
and suppose that $f$ is differentiable at point $u$. Since $f$ is order
preserving, the derivative $f'(u)$ can be identified to a nonnegative matrix. By homogeneity of $f$,
we have $f(su)=\lambda su$, and so, differentiating $s\mapsto f(su)$ at
$s=1$, we get $f'(u)u=\lambda u$. Hence, by the Perron-Frobenius
theorem, $\lambda$ is the spectral radius of $f'(u)$.
So we can list the eigenvalues of $f'(u)$
as $\lambda =\mu_1, \mu_2,\dots,\mu_N$, counting multiplicities,
with $|\mu_i|\leq \lambda $ for all $i\in [N]\setminus[1]$.
We set
\begin{align}
\vartheta:= \max_{i\in [N]\setminus[1]}\frac{|1+\mu_i\lambda^{-1}|}{2}
\label{e-def-theta}
\end{align}
As soon as $\lambda$ is a simple eigenvalue of $f'(u)$, we 
have $\mu_i\neq \lambda$ for all $i\in[N]\setminus\{1\}$, 
which, together with $|\mu_i|\leq 1$, entails that
$|1+\mu_i\lambda^{-1}|/2<1$. Then, the assumption
$\vartheta<1$ is satisfied under this simplicity condition.
The next theorem shows that
this entails the geometric convergence with rate $\vartheta$ of the 
projective Krasnoselskii-Mann iteration. 

\begin{theorem}\label{th-new}
Suppose that $f$ is a monotone positively homogeneous map $\R_+^N\to \R_+^N$
having a positive eigenvector $u$, normalized so that $\prod_{i\in [N]}u_i=1$,
suppose that $f$ is differentiable
at point $u$, let $\vartheta$ be defined by~\eqref{e-def-theta}
and suppose finally that $\vartheta<1$. 
Then, 
\begin{align}
 \limsup_{k\to\infty} (d_{T}(v^k,u))^{1/k}
\leq \vartheta \enspace .
\label{e-geom}
\end{align}
\end{theorem}
\begin{proof}
The proof idea is inspired by the analysis 
of the power algorithm in~\cite{FGH09}. The power
algorithm defines the sequence 
\[
w^{k+1}=\frac{f(w^k)}{G\big[ f(w^k) \big]} \enspace ,
\]
i.e., the difference with the projective Krasnoselskii-Mann iteration
is the damping in~\eqref{eq:krasn}. 
We showed in the the proof of \Cref{th-converge}
that the projective Krasno\-sel\-skii-Mann iteration,
after the change of variable $v^k=\exp(y^k)$, is equivalent 
to the iteration
\[
y^{k+1} = H(y^k) \qquad \text{where } H(z)=\frac{1}{2}(\hat{S}(z)+z) \enspace,
\]
$\hat{S}(z)={S}(z) - N^{-1}\scal{e}{S(z)} e$ and $S(z)=\log\circ f\circ \exp(z)$.
Let $y:=\log u$, and let $\delta(u)$ denote the diagonal
matrix with entries $u_1,\dots,u_N$.  
A simple computation
shows that the matrix $F:=S'(y)$ satisfies
\[
F= \lambda^{-1}\delta(u)^{-1}f'(u)\delta(u) 
\]
so that $M:=\hat{S}'(y)$ is given by 
\[
M= F - N^{-1}ee^\top F \enspace .
\]
From $f'(u)u=\lambda u$, we deduce that $Fe=e$.
It follows that $Me =0$. Moreover, it is shown in the proof of Corollary 5.2 of~\cite{FGH09} that the eigenvalues of $M$ are precisely
$0$, $\lambda^{-1}\mu_2,\dots,\lambda^{-1}\mu_N$. Hence,
the sequence $y^{k}$ satisfies $y^{k+1}=H(y^k)$ where
$H$ is a self-map of the space $X$,
with fixed point $y$, and $H'(y)$ has a spectral radius $\vartheta$.
By a standard argument (end of the proof of Corollary 5.2, {\em ibid.}),
it follows that there is a neighborhood $Y$ of $y$ such that
$\limsup_{k\to\infty} \|y^k-y\|^{1/k}\leq \vartheta$ if $y^0\in Y$.
However, we already showed in \Cref{th-converge} that $y^k$ does converges
to $y$ for every initial condition $y^0\in X$. Hence, we deduce
that $\limsup_{k\to\infty} \|y^k-y\|^{1/k}\leq \vartheta$ for all $y^0\in X$.
Since $v^k=\exp(y^k)$, we deduce that~\eqref{e-geom} holds.
\end{proof}

\begin{remark}\label{rk-shmuel}
\Cref{th-new} is easily applicable in situations in which the map
$f$ is known a priori to be differentiable, for instance when $f$ is a polynomial map
associated to a nonnegative tensor, as in~\cite{FGH09}. 
It is shown in~\cite{FGH09} that the power algorithm
converges with a geometric rate $\vartheta'=\max_{i\in [N]\setminus \{1\}}
|\lambda^{-1}\mu_i|$ as soon as $\vartheta'<1$. The condition
that $\vartheta'<1$ is more restrictive that $\vartheta<1$ as it
excludes the presence of a non-trivial peripheral spectrum of $f'(u)$.
Hence, \Cref{th-new} improves on results of~\cite{FGH09},
by showing that the Krasnoselskii-Mann iteration
does converge geometrically under more general circumstances than the power
algorithm.
\end{remark}
\begin{remark}
When $f$ is not everywhere differentiable, verifying
the assumption of \Cref{th-new} can be difficult.
In particular, the map $f=T^d$, defined as a finite supremum of linear maps, is differentiable
except at the exceptional points $x$ where the supremum
in~\eqref{e-def-Td} is achieved twice. To apply
\Cref{th-new} to the eigenproblem for $T^d$, we need to
know a priori that the eigenvector $u$ is a differentiability point
of $T^d$. For certain classes of maps, including max-plus linear
maps, this property can be shown to hold under some genericity assumptions,
exploiting methods in~\cite{1610.09651}. However, the map $T^d$ has an explicit
structured form which makes it hard to use such genericity arguments. 
\end{remark}
\begin{remark}
When $f$ is not everywhere differentiable, an easier
route to get a geometric convergence rate is to use 
the notion of {\em semidifferential} of $f$, as in~\cite{agn12}.
Thus, we suppose now that $f(u+x)=f(u)+f'_u(x)+o(\|x\|)$, where $f'_u$
is a continuous positively homogeneous map, not necessarily linear,
called the semidifferential of $f$ at point $u$. We refer
the reader to~\cite{agn12} for more background on semidifferentials.
It follows in particular from Theorem~3.8, {\em ibid.}, that the map $f=T^d$ has a semidifferential at every point. The power iteration
for a semidifferentiable monotone positively homogeneous
map is analysed in~\cite[Theorem~7.8]{agn12}. It is shown there
to converge with a geometric rate $\tilde{r}(f'_u)$ where
$\tilde{r}$, defined in the same reference,
is the spectral radius with respect to the
local norm attached to Hilbert's projective metric. 
We leave it to the reader to verify that
a modification of the proof of Theorem~7.8, {\em ibid.},
leads to the conclusion that the sequence generated by the projective
Krasnoselskii-Mann iteration satisfies
\begin{align}
\limsup_{k\to\infty} d_T(v^k,u)^{1/k}\leq \frac{1+\tilde{r}(f'_u)}{2} \enspace .
\label{e-semidiff}
\end{align}
When $f$ is 
differentiable, it can be checked that
$\tilde{r}(f'_u)=\vartheta'=\max_{i\in [N]\setminus \{1\}}|\lambda^{-1}\mu_i|$,
with the same notation than in \Cref{th-new} and in \Cref{rk-shmuel},
and so, in this case, the estimate of the convergence
rate provided by~\eqref{e-semidiff} can be coarser than the one provided by 
\Cref{th-new}.
However, in the nondifferentiable case, the assumption that $\tilde{r}(f'_u)<1$ is often easily verifiable, e.g., by Doeblin-type contraction arguments,
as in the final section of~\cite{agn12}.
\end{remark}

\section{Benchmarks}
\label{sec-benchmarks}

The present method has been implemented in OCaml and  has been run on one core of an $2.2$ GHz Intel Core i7 processor with $8$ Gb of RAM. We report two numerical experiments, showing respectively the convergence of the scheme and the gain in scalability.

\subsection{Convergence of the hierarchy}

We illustrate the convergent nature of the hierarchy on the pair of matrices
\begin{align*}
A = \begin{pmatrix}
0 & 1 & 0 & 0 & 0 \\
1 & 0 & 2 & 0 & 0 \\
0 & 0 & 1 & 0 & 0 \\
0 & 1 & 0 & 0 & 1 \\
0 & 0 & 0 & 2 & 1
\end{pmatrix}
\qquad
B = \begin{pmatrix}
1 & 0 & 2 & 0 & 0 \\
0 & 0 & 0 & 1 & 2 \\
1 & 1 & 0 & 0 & 0 \\
0 & 0 & 0 & 1 & 0 \\
0 & 1 & 0 & 0 & 0
\end{pmatrix} \,.
\end{align*}

By definition of the joint spectral radius, the spectral radius of a product of $N$ matrices in $\A$ is no larger that the $N$-th power of the joint spectral radius $\rho(\A)$. When such a product achieves equality, we say that the set $\A$ has a \emph{spectrum maximizing product}~\cite{gugl_mono} of length $N$.
The pair $\{A,B\}$ has a spectrum maximizing product of length $6$ given by $A^2B^4$ yielding a joint spectral radius equal to $2.0273$.

 We report in~\Cref{tab:bench1} the eigenvalue obtained by solving the hierarchy~\eqref{hier} for $1 \leq d \leq 9$ as well as the computation time.
We observe that the hierarchy is stationary at $d = 7$ and that we recover the exact value of the joint spectral radius. The last column indicates the relative error $\big[\lambda_d - \rho(\A)\big] / \rho(\A)$. Finally, we also observe the exponential cost in computation time at the level $d$ of the hierarchy.

\begin{table}
\centering
\begin{tabular}{cccc}
Level $d$ & CPU Time (s) & Eigenvalue $\lambda_d$ & Relative error \\
\hline
1 & $0.01$ & $2.165$ & $6.8$\% \\
2 & $0.01$ & $2.102$ & $3.7$\% \\
3 & $0.01$ & $2.086$ & $2.9$\% \\
4 & $0.01$ & $2.059$ & $1.6$\% \\
5 & $0.02$ & $2.041$ & $0.7$\% \\
6 & $0.05$ & $2.030$ & $0.1$\% \\
7 & $0.7$ & $2.027$ & $0.0$\% \\
8 & $0.32$ & $2.027$ & $0.0$\% \\
9 & $1.12$ & $2.027$ &$0.0$\% 
\end{tabular}
\caption{Convergence of the hierarchy on $5 \times 5$ matrices}
\label{tab:bench1}
\end{table}

\subsection{Scalability of the approach}

We demonstrate the scalability of our method on  quadruplets of matrices of increasing size, with random entries between $0$ and $0.9$. We show in~\Cref{tab:bench2} the computation time associated with each dimension.  The iteration process converges in less than $50$ iterations in all examples, with a $10^{-6}$ numerical stopping criterion. A monotone extremal hemi-norm has been computed as the supremum of $16$ or $64$ linear forms (respectively for $d = 2$ and $d = 3$). 

\begin{table}
\centering
\begin{tabular}{cccc}
Dimension $n$ & Level $d$  & Eigenvalue $\lambda_d$ & CPU Time \\
\hline
$10$ & $2$ & $4.287$ & $0.01$ s\\
& $3$ & $4.286$ & $0.03$ s \\
\hline
$20$ & $2$ & $8.582$ & $0.01$ s \\
& $3$ & $8.576$ & $0.03$ s\\
\hline
$50$ & $2$ & $22.34$ & $0.04$ s \\
& $3$ & $22.33$ & $0.16$ s\\
\hline
$100$ & $2$ & $44.45$ & $0.17$ s\\
& $3$ & $44.45$ & $0.53$ s\\
\hline
$200$ & $2$ & $89.77$ & $0.71$ s \\
& $3$ & $89.76$ &$2.46$ s \\
\hline
$500$ & $2$ & $224.88$ & $5.45$ s\\
& $3$ & $224.88$ & $19.7$ s\\
\hline
$1000$ & $2$ & $449.87$ & $44.0$ s\\
& $3$ & $449.87$ & $2.7$ min\\
\hline
$2000$ & $2$ & $889.96$ & $4.6$ min \\ 
& $3$ & $889.96$ & $19.2$ min\\
\hline
$5000$ & $2$ & $2249.69$ & $51.9$ min \\
& $3$ & $2249.57$ & $3.3$ h
\end{tabular}
\caption{Computation time for large matrices}
\label{tab:bench2}
\end{table}

\section{Conclusion}

We have proposed a new approach for computing a convergent sequence of upper bounds of the joint spectral radius of nonnegative matrices, by solving a hierarchy of non-linear eigenproblems. 
At any level of this hierarchy, the non-linear eigenvalue $\lambda$ provides an upper bound for the joint spectral radius, whereas the eigenvector encodes a
monotone $\lambda$-extremal norm. 
The non-linear eigenproblem is solved efficiently by a projective version of the Krasnoselskii-Mann iteration. We have implemented this approach and numerical results are witnesses of the scalability of this approach, compared to other works based on the solution of optimization problems.

We finally point out one open problem.
Guglielmi and Protasov showed in~\cite{gugl_mono} that 
when the joint spectral radius is obtained for a unique periodic product, and when this product has a unique dominant eigenvalue, then, there is a polyhedral invariant norm. Each level of the present hierarchy generates
a dictionary of linear forms, whose supremum yields 
a polyhedral extremal norm. This dictionary becomes
richer when the level of the hierarchy is increased.
Hence, we may ask whether the hierarchy is exact, i.e.,
whether there exist a level $d$ such that $r(T^d) = \rho(\A)$,
under the same assumption.

\bibliographystyle{alpha}
\newcommand{\etalchar}[1]{$^{#1}$}

\end{document}